\documentclass[12pt]{amsart}

\pdfoutput=1

\usepackage{a4wide}

\usepackage{enumerate}

\newtheorem{theorem}{Theorem}[section]

\newtheorem{proposition}[theorem]{Proposition}
\newtheorem{corollary}[theorem]{Corollary}
\newtheorem*{maintheorem}{Main Theorem}
\newtheorem*{theorem a}{Theorem A}
\newtheorem*{theorem b}{Theorem B}
\newtheorem*{theorem c}{Theorem C}
\newtheorem*{theorem d}{Theorem D}
\newtheorem*{theorem e}{Theorem E}

\theoremstyle{definition}
\newtheorem{remark}[theorem]{Remark}

\newcommand{\C}{\ensuremath{\mathbb{C}}}
\newcommand{\R}{\ensuremath{\mathbb{R}}}

\newcommand{\g}[1]{\ensuremath{\mathfrak{#1}}}
\newcommand{\cal}[1]{\ensuremath{\mathcal{#1}}}
\newcommand{\II}{\ensuremath{I\!I}}

\DeclareMathOperator{\spann}{span}

\allowdisplaybreaks

\begin{document}

\title[Strongly $2$-Hopf hypersurfaces in $\C P^2$ and $\C H^2$]{Strongly 
$2$-Hopf hypersurfaces\\ in complex projective and hyperbolic planes}

\author[J.~C.~D\'iaz-Ramos]{Jos\'e Carlos D\'iaz-Ramos}
\address{Department of Mathematics, Universidade de Santiago de Compostela, Spain.}
\email{josecarlos.diaz@usc.es}
\author[M.~Dom\'{\i}nguez-V\'{a}zquez]{Miguel Dom\'{\i}nguez-V\'{a}zquez}
\address{Instituto de Ciencias Matem\'aticas (CSIC-UAM-UC3M-UCM), Madrid, Spain.}
\email{miguel.dominguez@icmat.es}
\author[C.~Vidal-Casti\~neira]{Cristina Vidal-Casti\~neira}
\address{Department of Mathematics, Universidade de Santiago de Compostela, Spain.}
\email{cristina.vidal@usc.es}

\thanks{The second author has been supported by a Juan de la Cierva-formaci\'on 
fellowship (Spain) and by the ICMAT Severo Ochoa project SEV-2015-0554 (MINECO, Spain). All authors have been supported by projects EM2014/009, 
GRC2013-045 and MTM2013-41335-P with FEDER funds (Spain).}

\subjclass[2010]{53B25, 53C42, 57S15, 57S20}


\begin{abstract}
We give a geometric characterization of certain hypersurfaces of 
cohomogeneity one in the complex projective and hyperbolic planes. We also 
obtain some partial classifications of austere hypersurfaces and of Levi-flat 
hypersurfaces with constant mean curvature in these spaces.
\end{abstract}

\keywords{Complex projective plane, complex hyperbolic plane, cohomogeneity 
one, strongly $2$-Hopf, polar action, constant mean curvature, austere, ruled, 
Levi-flat.}

\maketitle

\section{Introduction}
The method of equivariant differential geometry has shown to be a powerful tool for the construction of submanifolds with specific geometric properties; see for example~\cite{HL}, \cite{Hs}. Given a proper isometric action of a Lie group $H$ on a Riemannian manifold $\bar{M}$, the idea of the method is to find a curve in the orbit space $\bar{M}/H$ such that the union of the corresponding orbits in $\bar{M}$ yields a submanifold $M$ with the desired geometric property. It turns out that for many interesting properties, finding such a curve is equivalent to solving certain ordinary differential equation. Thus, existence and uniqueness of such a curve is guaranteed, for given initial conditions. The resulting submanifolds $M$ are, intrinsically, manifolds of cohomogeneity one, that is, they admit an isometric action whose principal orbits have codimension one in $M$.

In~\cite{GG}, Gorodski and Gusevskii constructed many examples of complete constant mean curvature hypersurfaces of cohomogeneity one in complex hyperbolic spaces $\C H^n$, by applying the equivariant method to several cohomogeneity two polar actions on $\C H^n$. We recall that a proper isometric action on a Riemannian manifold is called polar if there is a submanifold intersecting all the orbits of the action perpendicularly; such a submanifold must be totally geodesic, and is called a section of the action. Thus, the resulting hypersurfaces appear as the union of orbits through some curve in the $2$-dimensional section.

Recently, the authors~\cite{DDV} discovered the first examples of real hypersurfaces with exactly two distinct nonconstant principal curvatures in the complex projective and hyperbolic planes, $\C P^2$ and $\C H^2$, thus answering an open question posed by Niebergall and Ryan in~\cite{NR}. These new examples are, again, constructed using the equivariant method applied to cohomogeneity two polar actions on $\C P^2$ and $\C H^2$. (Ivey and Ryan derived a construction of the same examples by a different approach in \cite{IR}.)

In the context of real hypersurfaces in K\"ahler manifolds, the class of Hopf hypersurfaces has been studied thoroughly. Recall that, if $M$ is a real hypersurface in a K\"ahler manifold with complex structure $J$, and $\xi$ is a (locally defined) unit normal vector field on $M$, then $J\xi$ is called the Hopf vector field of $M$. Moreover, $M$ is said to be Hopf at a point $p\in M$ if $J\xi$ is an eigenvector of the shape operator $S$ of $M$ at $p$, and $M$ is called a Hopf hypersurface if it is Hopf at every point. For example, all homogeneous hypersurfaces in $\C P^n$ (that is, those which are orbits of a cohomogeneity one isometric action on $\C P^n$) happen to be Hopf. Furthermore, Hopf hypersurfaces with constant principal curvatures in $\C P^n$ and $\C H^n$ have been classified~\cite{Be:crelle}, \cite{Ki}, and it follows from these classifications that such hypersurfaces are open parts of homogeneous ones.

However, the examples constructed in \cite{DDV} and \cite{GG} are generically non-Hopf. Moreover, in $\C H^n$, $n\geq 2$, there are examples of non-Hopf homogeneous hypersurfaces~\cite{BD09}. The observation that motivates this paper is that most of the examples in \cite{DDV} and \cite{GG}, and some examples in~\cite{BD09}, share the following geometric properties:
\begin{enumerate} [{\rm (C1)}]
\item The smallest $S$-invariant distribution $\cal{D}$ of $M$ that contains $J\xi$ has rank~$2$.
\item $\cal{D}$ is integrable.
\item The spectrum of $S\rvert_\cal{D}$ is constant along the integral submanifolds of $\cal{D}$.
\end{enumerate}
Here $S$ stands for the shape operator of $M$. A real hypersurface $M$ satisfying (C1) and (C2) was called \emph{$2$-Hopf} 
in~\cite{CR} and~\cite{IR}. Motivated by this terminology, we will say that a real 
hypersurface $M$ in a K\"ahler manifold is \emph{strongly $2$-Hopf} if it 
satisfies conditions (C1), (C2) and (C3) above. The generalization of these 
definitions to $k$-Hopf and strongly $k$-Hopf hypersurfaces, for any positive 
integer $k$, is straightforward. It is important to mention that the notions of 
Hopf, $1$-Hopf and strongly $1$-Hopf real hypersurfaces agree when the ambient 
manifold is a nonflat complex space form $\C P^n$ or $\C H^n$ (see~\cite{NR}). 
Also, note that condition (C1) has been studied in the context of real 
hypersurfaces with constant principal curvatures in nonflat complex space 
forms~\cite{DD:indiana}. Finally, observe that if we define $h$ as the number 
of principal curvature spaces of $M$ onto which the Hopf vector field has 
nontrivial projection, then $M$ is Hopf precisely when $h=1$, and condition 
(C1) is equivalent to $h=2$.

The main result of this paper is a characterization of the cohomogeneity one hypersurfaces in $\C P^2$ or $\C H^2$ constructed via the equivariant method applied to a polar action of cohomogeneity two. Such characterization is achieved in terms of the strongly 2-Hopf property. It is important to mention here that polar actions on nonflat complex space forms have been classified~\cite{DDK}, \cite{PT}: up to orbit equivalence, there is exactly one polar action of cohomogeneity two on $\C P^2$, and exactly four on $\C H^2$ (see Subsection~\ref{subsec:polar}). In what follows we will denote by $\bar{M}^2(c)$ a nonflat complex space form of complex dimension $2$ and constant holomorphic curvature $c\neq 0$. Then, our main result can be stated as follows. 

\begin{maintheorem} \label{th:strongly2Hopf}
	Consider a polar action of a group $H$ acting with cohomogeneity two and with section $\Sigma$ on a nonflat complex space form $\bar{M}^2(c)$.
	
	Let $p\in\Sigma$ be a regular point, and $\sigma\colon(-\varepsilon,\varepsilon)\to\Sigma$ a unit speed curve in $\Sigma$ with $\sigma(0)=p$. Define the subset $H\cdot \sigma=\{h(\sigma(t)):h\in H,\, t\in(-\varepsilon,\varepsilon)\}$ of $\bar{M}^2(c)$. Then, there exists a finite subset $\g{w}_p$ of the unit sphere of $T_p\Sigma$ such that, for $\varepsilon$ small enough, if $\dot{\sigma}(0)\notin\g{w}_p$, the set $H\cdot \sigma$ is a strongly 2-Hopf hypersurface of $\bar{M}^2(c)$, whereas if $\dot{\sigma}(0)\in\g{w}_p$, then $H\cdot \sigma$ is a real hypersurface of $\bar{M}^2(c)$ that is Hopf at $p$.
	
	Conversely, any strongly $2$-Hopf real hypersurface in $\bar{M}^2(c)$ is locally congruent to a hypersurface constructed as above.
\end{maintheorem}

A first consequence of this result is a local characterization of the examples of constant mean curvature hypersurfaces constructed by Gorodski and Gusevskii~\cite{GG}.

\begin{corollary} \label{cor:CMC}
	Let $H$ and $\Sigma$ be as in the Main Theorem, and let $\eta\in\R$. Then, for any regular point $p\in \Sigma$ and any unit $w\in T_p \Sigma$, there is exactly one locally defined curve $\sigma$ on $\Sigma$ with $\sigma(0)=p$, $\dot{\sigma}(0)=w$, and such that the hypersurface $H\cdot \sigma$ has constant mean curvature $\eta$. 	Conversely, any strongly $2$-Hopf real hypersurface with constant mean curvature in $\bar{M}^2(c)$ is locally congruent to a hypersurface constructed in this way.
\end{corollary}

It is interesting to point out here that, in the family of constant mean 
curvature hypersufaces in $\bar{M}^2(c)$, the wealth of strongly 2-Hopf 
examples contrasts with the rigidity of those that are Hopf. Indeed, we have 
the following result: 

\begin{theorem} \label{th:Hopf}
	Let $M$ be a connected Hopf real hypersurface in $\C P^2$ or $\C H^2$ with constant mean curvature. Then $M$ is an open part of a homogeneous Hopf hypersurface.
\end{theorem}

The homogeneous Hopf hypersurfaces in $\C P^n$ and $\C H^n$ are usually referred to as the examples in Takagi's and Montiel's lists~\cite{NR}. In the case of $\C P^2$ these are geodesic spheres and tubes around a totally geodesic $\R P^2$, whereas in $\C H^2$ they are geodesic spheres, tubes around a totally geodesic $\R H^2$, tubes around a totally geodesic $\C H^1$, and horospheres. 

We will also investigate the so-called austere hypersurfaces. These objects 
were introduced by Harvey and Lawson~\cite{HL:acta} in their study of special Lagrangian submanifolds, and are defined as those 
hypersurfaces whose principal curvature functions are invariant under 
multiplication by $-1$. Thus, austere hypersurfaces provide a subclass of minimal hypersurfaces. The classification of austere hypersurfaces in 
spheres, or in the complex projective and hyperbolic planes, is not known 
\cite{CK}, \cite{II}. In this sense we prove the following result.

\begin{theorem}\label{th:austere}
	Let $M$ be a real hypersurface of $\bar{M}^2(c)$, $c\neq 0$, whose Hopf 
	vector field has nontrivial projection onto at most two principal curvature 
	spaces (i.e.\ $h\leq 2$). Then $M$ is austere if and only if it is an open 
	part of one of the following examples:
	\begin{enumerate} [\rm (i)]
		\item a Lohnherr hypersurface in $\C H^2$, or
		
		\item a Clifford cone in $\C P^2$ or $\C H^2$, or
		
		\item a bisector in $\C H^2$.	
	\end{enumerate} 
	In particular, $M$ is strongly $2$-Hopf on the open and dense subset of nonumbilical points.
\end{theorem}

All the examples in this classification are ruled, in the sense that their maximal complex distribution is integrable and its integral submanifolds are totally geodesic in the ambient space. We briefly describe the examples in Theorem~\ref{th:austere}. The Lohnherr 
hypersurface is the only, up to congruence, complete ruled hypersurface of $\C H^n$ with 
constant principal curvatures~\cite{LR}. It is also the unique minimal homogeneous hypersurface of $\C 
H^n$ \cite{BD09}. A Clifford cone is a minimal hypersurface which is 
constructed as follows  (see also~\cite{ALS}, \cite{Goldman} and~\cite{Ki:mathann} for alternative descriptions). The Lie group $H=U(1)\times U(1)$ acts on 
$\bar{M}^2(c)$ polarly with cohomogeneity two. This action has three fixed 
points in $\C P^2$, and only one in $\C H^2$. Let $p$ be one of these fixed 
points, and $S^r$ any geodesic sphere centered at $p$. 
Then a Clifford cone with vertex $p$ is the (singular) hypersurface made of all 
geodesic rays starting from $p$ and hitting the only $2$-dimensional 
$H$-orbit that is minimal as a submanifold of $S^r$. Finally, a bisector in $\C H^n$ is a minimal hypersurface 
of cohomogeneity one defined as the set of points in $\C H^n$ that are at the 
same distance from two fixed points~\cite{Goldman}.

Another application of the Main Theorem concerns the existence of Levi-flat 
hypersurfaces of cohomogeneity one. We recall that a real hypersurface of a 
complex manifold is called Levi-flat if it is foliated by complex hypersurfaces (see 
\S\ref{subsec:strongly2HopfLevi-flat}). This notion is important in the study 
of holomorphic foliations, and indeed, an outstanding problem is the existence 
of complete, smooth Levi-flat hypersurfaces in the complex projective plane; 
nonexistence has been proved for $\C P^n$, $n\geq 3$ \cite{Siu}. Note that the 
following result contrasts with the nonexistence of Levi-flat, Hopf real 
hypersurfaces in nonflat complex space forms~\cite{Cho}.

\begin{corollary} \label{cor:Levi-flat}
	Let $H$ and $\Sigma$ be as in the Main Theorem. Then, for any regular point $p\in \Sigma$ and any unit $w\in T_p \Sigma$, there is exactly one locally defined curve $\sigma$ on $\Sigma$ with $\sigma(0)=p$, $\dot{\sigma}(0)=w$, and such that the hypersurface $H\cdot \sigma$ is Levi-flat. Conversely, any strongly $2$-Hopf, Levi-flat real hypersurface in $\bar{M}^2(c)$ is constructed locally in this way.
\end{corollary}

It is interesting to determine to what extent imposing some additional geometric conditions restricts the class of Levi-flat hypersurfaces. In this sense, Bryant~\cite{Bryant} classified Levi-flat minimal hypersurfaces in $2$-dimensional complex space forms. It follows from his result that, for $\C P^2$ and $\C H^2$, each example in his classification is invariant under a one-dimensional subgroup of the ambient isometry group. By weakening the minimality condition, and adding the strongly 2-Hopf assumption, we can obtain the following result.

\begin{theorem}\label{th:strongly2HopfCMCLevi-flat}
Let $M$ be a connected, Levi-flat, strongly $2$-Hopf real hypersurface in $\bar{M}^2(c)$, $c\neq 0$. Then $M$ has constant mean curvature if and only if it is an open part of 
\begin{enumerate} [\rm (i)]
	\item a Lohnherr hypersurface in $\C H^2$, or
	
	\item a Clifford cone in $\C P^2$ or $\C H^2$, or
	
	\item a bisector in $\C H^2$.	
\end{enumerate} 
In particular, $M$ is austere and ruled.
\end{theorem}

This work is organized as follows. In Section~\ref{sec:preliminaries} we 
establish notation, recall some basic concepts and results about submanifold geometry and polar actions
on complex space forms, and prove Theorem~\ref{th:Hopf}. In Section~\ref{sec:LeviCivita} we prove some formulas for the Levi-Civita connection of a hypersurface satisfying $h=2$. The proof of the Main Theorem is presented in 
Section~\ref{sec:strongly2Hopf}: in \S\ref{subsec:construction} we explain how to construct strongly 2-Hopf hypersurfaces, and 
in~\S\ref{subsec:strongly2Hopf} we characterize these examples. Then, Section~\ref{sec:austere} is devoted to the study of austere hypersurfaces and the proof of Theorem~\ref{th:austere}. Finally, in 
Section~\ref{sec:aplications}, we give some applications of the Main Theorem and 
prove the remaining theorems.

\section{Preliminaries}\label{sec:preliminaries}

In this section we settle some notation and terminology concerning submanifold theory and polar actions, with particular emphasis on the case of nonflat complex space forms.

\subsection{Submanifold geometry in complex space forms}\label{subsec:submanifold}

Let $M$ be a smooth submanifold of a Riemannian manifold $\bar{M}$. Since the arguments that follow are local, we can assume that $M$ is embedded. 
We denote by $T_{p}M$ and $\nu_{p}M$ the tangent and normal spaces to $M$ at $p$ respectively. Let $X$, $Y$, $Z$, $W$ be tangent vector fields along $M$, and let $\xi$ be normal. We denote by $\langle\,\cdot\,,\,\cdot\,\rangle$ the metric of $\bar{M}$, by $\bar{\nabla}$ its Levi-Civita connection, and by $\bar{R}$ its curvature tensor, which we adopt with the following sign convention: $\bar{R}(X,Y)Z=[\bar{\nabla}_X,\bar{\nabla}_Y]Z-\bar{\nabla}_{[X,Y]}Z$. The Levi-Civita connection of $M$ is denoted by $\nabla$, and is determined by the Gauss formula
\[
\bar{\nabla}_X Y=\nabla_X Y+\II(X,Y),
\]
where $\II$ is the second fundamental form of~$M$. The Weingarten formula reads
\[
\bar{\nabla}_X \xi=-S_\xi X+\nabla^\perp_X \xi
\]
where $S_\xi$ is the shape operator of $M$ with respect to $\xi$, and $\nabla^\perp$ is the normal connection of $M$. Moreover, we have the relation $\langle \II(X,Y), \xi \rangle=\langle S_{\xi}X,Y \rangle$.

The shape operator $S_\xi$ is a self-adjoint endomorphism with respect to the induced metric on $M$, and thus it can be diagonalized with real eigenvalues. These eigenvalues are called the principal curvatures of $M$ with respect to $\xi$, the corresponding eigenspaces are the principal curvature spaces, and the corresponding eigenvectors are the principal curvature vectors. The mean curvature vector field $\cal{H}$ of $M$ is defined as the trace of the second fundamental form. We say that $M$ has parallel second fundamental form (resp.\ parallel mean curvature) if $\II$ (resp.\ $\cal{H}$) is parallel with respect to the normal connection $\nabla^\perp$. We say that $M$ has flat normal bundle if every normal vector can be extended locally to a parallel normal vector field or, equivalently, if the curvature of $\nabla^\perp$ is zero.

Now let $M$ be a hypersurface of $\bar{M}$, and $\xi$ a unit normal vector field on $M$. In this case we simply write $S$ for the shape operator $S_\xi$. The Codazzi equation is then written as
\[
\langle\bar{R}(X,Y)Z,\xi\rangle= \langle(\nabla_X S)Y,Z\rangle-\langle(\nabla_Y S)X,Z\rangle, 
\]
and, by letting $R$ denote the curvature tensor of $M$, the Gauss equation reads
\[
\langle\bar{R}(X,Y)Z,W\rangle=\langle{R}(X,Y)Z,W\rangle +\langle SX,Z\rangle\langle SY,W\rangle-\langle SX,W\rangle\langle SY,Z\rangle.
\]

We now restrict our attention to the case $\bar{M}= \bar{M}^n(c)$, where $\bar{M}^n(c)$ represents a complex space form of complex dimension $n$ and constant holomorphic curvature $c\in \R$, that is, a complex projective space $\C P^n$ if $c>0$, a complex Euclidean space $\C^n$ if $c=0$, or a complex hyperbolic space $\C H^n$ if $c<0$. We denote by $J$ the complex structure of $\bar{M}^n(c)$. Since $\bar{M}^n(c)$ is K\"ahler, we have that $\bar{\nabla}J=0$. We will also need the formula of the curvature tensor $\bar{R}$ of a complex space form of constant holomorphic sectional curvature $c$:
\begin{align*}
	\langle\bar{R}(X,Y)Z,W\rangle={}&
	\frac{c}{4}\Bigl(
	\langle Y,Z\rangle\langle X,W\rangle
	-\langle X,Z\rangle\langle Y,W\rangle\\[-1ex]
	&\phantom{\frac{c}{4}\Bigl(}
	+\langle JY,Z\rangle\langle JX,W\rangle
	-\langle JX,Z\rangle\langle JY,W\rangle
	-2\langle JX,Y\rangle\langle JZ,W\rangle
	\Bigr).
\end{align*}

Let $M$ be a real hypersurface of $\bar{M}^n(c)$, that is, a submanifold with real codimension one. The tangent vector field $J\xi$ is called the Hopf or Reeb vector field of $M$. We define the integer-valued function $h$ on $M$ as the number of principal curvature spaces onto which $J\xi$ has nontrivial projection or, equivalently, as the dimension of the minimal subspace of the tangent space to $M$ that contains $J\xi$ and is invariant under the shape operator $S$. Thus, $M$ is said to be Hopf at a point $p$ if $h(p)=1$, and is called a Hopf hypersurface if $h=1$ on $M$, that is, if $J\xi$ is a principal curvature vector field everywhere. If $h$ is constantly equal to an integer number $k$, then there is a smooth distribution $\cal{D}$ of rank $k$ on $M$ that consists of the minimal subspace of the tangent space to $M$ at each point that contains $J\xi$ and is $S$-invariant. If $\cal{D}$ is integrable, then $M$ is said to be $k$-Hopf. If additionally, the principal curvatures of $M$ corresponding to the principal directions in $\cal{D}$ are constant along the leaves of $\cal{D}$, then we will say that $M$ is strongly $k$-Hopf.

For more information on submanifold geometry and real hypersurfaces in complex space forms we refer to~\cite{BCO03}, \cite{CR} and \cite{NR}.

\medskip

As an application of well-known results about Hopf real hypersurfaces in nonflat complex space forms, we prove Theorem~\ref{th:Hopf}.

\begin{proof}[Proof of Theorem~\ref{th:Hopf}]
	Let $M$ be a Hopf real hypersurface in $\bar{M}^2(c)$, $c\neq 0$, with constant mean curvature. Let $\alpha$ denote the principal curvature of the Hopf vector field. By \cite[Theorem~2.1]{NR} we know that $\alpha$ is constant on $M$. Now, by \cite[Corollary~2.3(ii)]{NR}, if $\beta$ and $\gamma$ denote the other principal curvatures of $M$, we have that $2\alpha(\beta+\gamma)-4\beta\gamma+c=0$. This equation, together with the constancy of $\alpha$ and $\alpha+\beta+\gamma$, implies that $\beta$ and $\gamma$ are also constant. Hence, $M$ is a Hopf hypersurface with constant principal curvatures in $\bar{M}^2(c)$, $c\neq 0$. According to their classification by Kimura~\cite{Ki} and Berndt~\cite{Be:crelle}, we conclude that $M$ must be an open part of a homogeneous Hopf hypersurface.
\end{proof}

\subsection{Polar actions}\label{subsec:polar}
Let $\bar{M}$ be a Riemannian manifold, and $H$ a connected group of isometries of $\bar{M}$. The isometric action $H\times \bar{M}\to\bar{M}$, $(h,p)\mapsto h(p)$, is called proper if the map $H\times\bar{M} \to \bar{M}\times\bar{M}$, $(h, p)\mapsto (g(p), p)$, is proper, which implies that the orbits of the $H$-action are embedded, the space $\bar{M}/H$ of orbits is Hausdorff, and the isotropy groups $H_p=\{h\in H: h(p)=p\}$ are compact. An orbit of a proper action is called principal if its isotropy groups are minimal among all the isotropy groups of the orbits. In particular, principal orbits have maximal dimension. The codimension of a principal orbit is called the cohomogeneity of the action. If an orbit has codimension higher than the cohomogeneity, it is called singular. A point is said to be regular if it lies on a principal orbit. 

Two isometric actions are called orbit equivalent if they have the same orbits, modulo an isometry of the ambient space. A submanifold of $\bar{M}$ is called homogeneous if it is an orbit of an isometric action on $\bar{M}$. A Riemannian manifold is said to be of cohomogeneity one if it admits a cohomogeneity one isometric action.

We say that a proper isometric action $H\times \bar{M}\to\bar{M}$ is polar if there exists a submanifold $\Sigma$ of $\bar{M}$ that intersects all the $H$-orbits, and every such intersection is perpendicular. Such a submanifold $\Sigma$ is totally geodesic, has the dimension of the cohomogeneity of the action, and is called a section. Polar actions admit sections through any given point in $\bar{M}$. It turns out that the set $\Sigma_{reg}$ of regular points in $\Sigma$ is an open and dense subset of $\Sigma$. For more information on isometric and polar actions we refer to~\cite[Chapter~2]{BCO03}.

In this work we use polar actions to construct examples of interesting hypersurfaces in nonflat complex space forms of dimension two. Thus, let us comment on their classification. Polar actions on complex projective spaces were classified (up to orbit equivalence) by Podest\`a and Thorbergsson~\cite{PT}, whereas for complex hyperbolic spaces the corresponding classification was obtained by the first two authors and Kollross~\cite{DDK}, although the case of the complex hyperbolic plane had previously been solved by Berndt and the first author~\cite{BD:agag}. We are interested in the case of cohomogeneity two. It is known that, in this case, sections $\Sigma$ are totally real, that is, $\langle JT\Sigma,T\Sigma\rangle=0$. In particular, they are totally geodesic real projective planes $\R P^2$ if $c>0$, and totally geodesic real hyperbolic planes $\R H^2$ if $c<0$. 

In $\C P^2$ there is only one polar action of cohomogeneity two up to orbit equivalence, namely the action of the group $U(1)\times U(1)$, which is induced from the standard action of $U(1)\times U(1)\times U(1)$ on the $5$-sphere via the Hopf map. This action has three fixed points, the other orbits are contained in the geodesic spheres around each one of these points, and topologically they can be circles or $2$-dimensional tori (the latter are the principal orbits). 

In $\C H^2$ there are four polar actions of cohomogeneity two up to orbit equivalence. One of them is dual to the one described for $\C P^2$. It is the action of $U(1)\times U(1)$ on $\C H^2$, which has only one fixed point in this case, and the other orbits are again circles or $2$-tori contained in the geodesic spheres around the fixed point. In order to describe the other three examples we introduce some notation (see~\cite{BD:agag} for details). Let $\g{g}=\g{su}(1,2)$ be the Lie algebra of the isometry group of $\C H^2$, and $\g{k}=\g{s}(\g{u}(1)\oplus\g{u}(2))$ the Lie algebra of the isotropy group of some point of $\C H^2$. The corresponding Cartan decomposition can be written as $\g{g}=\g{k}\oplus\g{p}$, where $\g{p}$ is the orthogonal complement of $\g{k}$ in $\g{g}$ with respect to the Killing form of $\g{g}$. Then, a choice of a maximal abelian subspace $\g{a}$ of $\g{p}$ determines a decomposition $\g{g}=\g{g}_{-2\alpha}\oplus\g{g}_{-\alpha}\oplus\g{g}_{0} \oplus\g{g}_{\alpha}\oplus\g{g}_{2\alpha}$, called the restricted root space decomposition. Here, $\g{g}_0=\g{k}_0\oplus\g{a}$, where $\g{k}_0\cong\g{u}(1)$ is the centralizer of $\g{a}$ in $\g{k}$. Thus, the other three cohomogeneity two polar actions on $\C H^2$ correspond to the connected subgroups $H$ of $SU(1,2)$ with the following Lie algebras: $\g{h}=\g{g}_0$, $\g{h}=\g{k}_0\oplus\g{g}_{2\alpha}$, and $\g{h}=\ell\oplus\g{g}_{2\alpha}$, where~$\ell$ is a $1$-dimensional vector subspace of $\g{g}_\alpha$. Topologically, the principal orbits of these first two actions are $2$-dimensional cylinders, while those of the last one are $2$-dimensional planes.

\section{Levi-Civita connection of a hypersurface with $h=2$}\label{sec:LeviCivita}

In this section we calculate the Levi-Civita connection of a real hypersurface $M$ in $\bar{M}^2(c)$, $c\neq 0$, satisfying $h=2$. This information will be used several times throughout this paper.

Let $M$ be a real hypersurface with unit normal vector field $\xi$ and shape operator $S$ in a nonflat complex space form $\bar{M}^2(c)$. Let $\alpha$, $\beta$ and $\gamma$ be the three principal curvatures of $M$. For each principal curvature $\lambda$, we denote by $T_\lambda$ the corresponding principal curvature distribution; note that, in principle, this distribution might be singular. We will denote by $\Gamma(T_\lambda)$ the module of smooth vector fields $X$ on $M$ such that $X_p\in T_\lambda(p)$ for every point $p$. 

For the following proposition we only assume that $M$ satisfies condition (C1) in the definition of strongly 2-Hopf hypersurface, that is, the Hopf vector field $J\xi$ of $M$ has nontrivial projections onto exactly $h=2$ principal curvature spaces, say onto $T_\alpha$ and $T_\beta$. This implies that  $\alpha \neq \beta$ at every point. 

\begin{proposition}\label{prop:algebra}
	There are positive smooth functions $a$, $b\colon M\to\R$ with $a^2+b^2=1$, and an orthonormal frame $\{U,V,A\}$ on $M$ with $U\in\Gamma(T_\alpha)$, $V\in\Gamma(T_\beta)$, $A\in\Gamma(T_\gamma)$, such that
\[
		J\xi = aU+bV, \qquad JU  = -bA-a\xi, \qquad JV  = aA-b\xi, \qquad JA=bU-aV.
\]
\end{proposition}

\begin{proof}
	Since $J\xi$ is a unit vector field tangent to $M$ which has nontrivial projection onto $T_\alpha$ and $T_\beta$, we can write $J\xi=a U+b V$, where $U\in\Gamma(T_\alpha)$, $V\in\Gamma(T_\beta)$ are unit vector fields, and $a$, $b$ are smooth functions on $M$ satisfying $a^2+b^2=1$, and $a,b>0$. Let $A\in\Gamma(T_\gamma)$ be a unit vector field; take it perpendicular to $U$ and $V$ in case $\gamma$ has multiplicity $2$. Then, $\{U,V,A\}$ constitutes an orthonormal frame on $M$.
	
	As $-\xi=J^2\xi=aJU+bJV$, and $a\neq 0$, taking inner product with $V$ we get that $\langle JU,V\rangle=0$. This implies that $JU$, $JV\in\spann\{A,\xi\}$. Now, $\langle JU,\xi\rangle=-\langle U,J\xi \rangle=-a$, and since $U$ has unit length, we obtain $\langle JU,A\rangle=\pm b$. By changing the sign of $A$ if necessary, we can assume that $JU=-bA-a\xi$. A similar argument shows that $JV=aA-b\xi$. Finally, these expressions  imply $\langle JA,U\rangle=b$, $\langle JA,V\rangle=-a$, and $\langle JA,\xi\rangle=0$, from where the result follows.
\end{proof}

\begin{proposition}\label{prop:Levi-Civita}
	Assume that $\alpha\neq\beta\neq\gamma\neq \alpha$ at every point. Then the Levi-Civita connection of $M$ in terms of the basis $\{U,V,A\}$ is given by the following equations:
	\begin{align*}
		\nabla_U U  &=\frac{ V\alpha }{\alpha -\beta }V-\frac{ 3 a b c-4 A\alpha }{4 (\alpha -\gamma )}A,
		&
		\nabla_U V  &=-\frac{V\alpha}{\alpha-\beta}U+\Bigl(\alpha+\frac{3a^2bc-4aA\alpha}{4b(\alpha-\gamma)}\Bigr)A,
		\\
		\nabla_V V  &=-\frac{U\beta}{\alpha-\beta}U+\frac{3abc+4A\beta}{4(\beta-\gamma)}A,
		&
		\nabla_V U  &=\frac{U\beta}{\alpha-\beta}V-\Bigl(\beta+\frac{3ab^2c+4bA\beta}{4a(\beta-\gamma)}\Bigr)A,
		\\
		\nabla_A U  &=\Bigl(\gamma - \frac{Ab}{a}\Bigr)V+\frac{U\gamma}{\alpha-\gamma}A,
		&
		\nabla_U A  &=\frac{3abc-4A\alpha}{4(\alpha-\gamma)}U-\Bigl(\alpha+\frac{3a^2bc-4aA\alpha}{4b(\alpha-\gamma)}\Bigr)V,
		\\
		\nabla_A V  &=\Bigl(-\gamma + \frac{Ab}{a}\Bigr)U+\frac{V\gamma}{\beta-\gamma}A,
		&
		\nabla_V A  &=\Bigl(\beta+\frac{3ab^2c+4bA\beta}{4a(\beta-\gamma)}\Bigr)U-\frac{3abc+4A\beta}{4(\beta-\gamma)}V,
		\\
		\nabla_A A  &=-\frac{U\gamma}{\alpha-\gamma}U-\frac{V\gamma}{\beta-\gamma}V.
	\end{align*}
	Moreover:
	\begin{equation}\label{eq:derivatives1}
	\begin{aligned}
		Ua &=  \frac{b V\alpha}{\alpha-\beta},
		&
		Va &= \frac{b U\beta}{\alpha-\beta},
		&
		Aa &= -\frac{b Ab}{a},\\
		Ub &= -\frac{a V\alpha}{\alpha-\beta},
		&
		Vb &=-\frac{a U\beta}{\alpha-\beta},
		&
		V\gamma &=\frac{a(\gamma-\beta)U\gamma}{b(\alpha-\gamma)},
	\end{aligned}
	\end{equation}
	\vspace{-1ex}
\begin{equation}\label{eq:derivatives2}
\small{
		\begin{aligned} 
			Ab =&{}\; a \gamma+\frac{a c \left(a^2-2 b^2\right)}{4 (\alpha -\beta
				)}-\frac{3 a^3 c (\beta -\gamma )}{4 (\alpha -\beta ) (\alpha -\gamma )}-\frac{\alpha  a (\beta-\gamma )}{\alpha -\beta }+\frac{a^2 (\beta -\gamma )}{b (\alpha -\beta ) (\alpha -\gamma )}A\alpha,\\
			A\beta =&{}-\frac{3abc}{4}-\frac{a\beta(\beta-\gamma)}{b}-\frac{ac(\beta-\gamma)}{4b(\alpha-\gamma)}-\frac{a\alpha(\beta-\gamma)^2}{b(\alpha-\gamma)}-\frac{3a^3c(\beta-\gamma)^2}{4b(\alpha-\gamma)^2}+\frac{a^2(\beta-\gamma)^2}{b^2(\alpha-\gamma)^2}A\alpha.
		\end{aligned}
	}
\end{equation}
	
\end{proposition}
\begin{proof}
	Using the fact that $U$ and $A$ are orthogonal eigenvectors of $S$ associated with the eigenvalues $\alpha$ and $\gamma$ respectively, we get
	\begin{align*}
		\langle(\nabla_U S)A,U\rangle
		&{}=\langle\nabla_U SA-S\nabla_UA,U\rangle
		=\langle\nabla_U(\gamma A),U\rangle-\langle \nabla_UA,SU\rangle\\
		&{}=(U\gamma)\langle A,U\rangle+\gamma\langle\nabla_UA,U\rangle
		-\alpha\langle\nabla_UA,U\rangle
		=(\alpha-\gamma)\langle\nabla_UU,A\rangle.
	\end{align*}
	As $U$ is a unit vector field we have $\langle\nabla_AU,U\rangle=0$. Thus, proceeding as before, we get $\langle(\nabla_A S)U,U\rangle=A\alpha$. Moreover, the expression of the curvature tensor of a complex space form yields $\langle\bar{R}(U,A)U,\xi\rangle=-3abc/4$. Hence, the Codazzi equation applied to the triple $(U,A,U)$ implies
	\[
	\langle\nabla_UU,A\rangle=\frac{1}{\alpha-\gamma}
	\left(A\alpha-\frac{3abc}{4}\right).
	\]
	
	Applying the Codazzi equation to the triples $(U,V,U)$, $(U,A,U)$, $(U,A,A)$, $(U,V,V)$, $(V,A,V)$ and $(V,A,A)$, we obtain in a similar way:
\begin{equation} \label{eq:nablas1}
	\begin{aligned}
	\langle\nabla_U U,V\rangle &{}=\frac{V\alpha}{\alpha-\beta},
	&
	\langle\nabla_U U,A\rangle &{}=\frac{1}{\alpha-\gamma}\Bigl(A\alpha-\frac{3abc}{4}\Bigr),
	&
	\langle\nabla_A A,U\rangle &=-\frac{U\gamma}{\alpha-\gamma},\\
	\langle\nabla_V V,U\rangle &{}=-\frac{U\beta}{\alpha-\beta},
	&
	\langle\nabla_V V,A\rangle &=\frac{1}{\beta-\gamma}\Bigl(A\beta+\frac{3abc}{4}\Bigr),
	&	
	\langle\nabla_A A,V\rangle
	&=-\frac{V\gamma}{\beta-\gamma}.
	\end{aligned}
	\end{equation}
	
	Since $J$ is parallel with respect to the Levi-Civita connection $\bar\nabla$ of $\bar{M}^2(c)$, we have $\bar{\nabla}_U J\xi=J\bar{\nabla}_U\xi=-JSU=-\alpha JU$. Taking this into account, and using Proposition~\ref{prop:algebra} and~\eqref{eq:nablas1}, we get
	\begin{align*}
		0\!&=U\langle A,J\xi\rangle=\langle\bar{\nabla}_U A,J\xi\rangle+\langle A,\bar{\nabla}_U J\xi\rangle
		=a\langle\nabla_U A,U\rangle+b\langle\nabla_U A,V\rangle+\alpha b\langle A,A\rangle+\alpha a\langle A,\xi\rangle\\
		&=-\frac{a}{\alpha-\gamma}\Bigl(A\alpha-\frac{3abc}{4}\Bigr)+b\langle\nabla_U A,V\rangle+\alpha b,
	\end{align*}
	from where we can obtain $\langle\nabla_U A,V\rangle$. This, and analogous calculations with $V\langle A,J\xi\rangle=0$ and $A\langle V,J\xi\rangle=Ab$, give the expressions
	\begin{equation}\label{eq:nablas2}
	\begin{aligned}
	\langle\nabla_U V,A\rangle &{}=\alpha-\frac{a}{b(\alpha-\gamma)}\bigl(A\alpha-\frac{3abc}{4}\bigr),
	&
	\langle\nabla_A U,V\rangle &{}=\gamma-\frac{Ab}{a},\\
	\langle\nabla_V U,A\rangle &{}=-\Bigl(\beta+\frac{b}{a(\beta-\gamma)}\bigl(A\beta+\frac{3abc}{4}\bigr)\Bigr).
	\end{aligned}
	\end{equation}
Equations~\eqref{eq:nablas1} and~\eqref{eq:nablas2} give the formulas for the Levi-Civita connection.

Now, the relations $U\langle V,J\xi\rangle=Ub$, $V\langle V,J\xi\rangle=Vb$, $A\langle A,J\xi\rangle=0$, $a^2+b^2=1$, together with analogous calculations as above and \eqref{eq:nablas1}, yield Equations \eqref{eq:derivatives1}.

Finally, if we apply the Codazzi equation to the triples $(U,V,A)$ and $(U,A,V)$, we obtain:
\begin{align*}
	\langle \nabla_V U, A\rangle &{}=\frac{c+4(\beta-\gamma)\langle \nabla_U V, A \rangle}{4(\alpha-\gamma)},
	&
	\langle \nabla_A U, V\rangle &{}=-\frac{(a^2-2b^2)c-4(\beta-\gamma)\langle \nabla_U V, A\rangle}{4(\alpha-\beta)}.
\end{align*}
Combining this with \eqref{eq:nablas2} we derive~\eqref{eq:derivatives2}.
\end{proof}

\section{Strongly 2-Hopf hypersurfaces}\label{sec:strongly2Hopf}
In this section we investigate the structure of strongly $2$-Hopf hypersurfaces in $\C P^2$ and $\C H^2$. We prove the first part of the Main Theorem in \S\ref{subsec:construction}, and the second part in \S\ref{subsec:strongly2Hopf}.

\subsection{Construction}\label{subsec:construction}

We proceed with the construction of the examples of strongly $2$-Hopf hypersurfaces in a nonflat complex space form $\bar{M}^2(c)$, $c\neq 0$. 

We fix a connected group $H$ of isometries of $\bar{M}^2(c)$ acting polarly and with cohomogeneity two on $\bar{M}^2(c)$. Let $\Sigma$ be a section for this action, and $\Sigma_{reg}$ the set of regular points of $\Sigma$.

Let $\sigma\colon t\in (-\varepsilon, \varepsilon)\mapsto \sigma(t)\in \Sigma_{reg}$ be a unit speed curve, and put $p=\sigma(0)$. Then, the subset
\[
M=H\cdot \sigma = \{h(\sigma(t)): t\in(-\varepsilon, \varepsilon), \, h\in H\}
\]
is a $3$-dimensional hypersurface in $\bar{M}^2(c)$ that is foliated by equidistant principal $H$-orbits, and orthogonally, by the curves $h\circ \sigma\colon t\in (-\varepsilon, \varepsilon)\mapsto (h\circ \sigma)(t)=h(\sigma(t))\in \Sigma_{reg}$, for each $h\in H$. Note that $M=H\cdot\sigma$ is intrinsically a cohomogeneity one manifold. Moreover, the integrable distributions associated with these two foliations are invariant under the shape operator of $M$. Indeed, if $\xi$ is a unit normal vector field on $M$, the principal curvatures (resp.\ principal curvature spaces) of some orbit $H\cdot q$ at $q$ with respect to $\xi$ are also principal curvatures (resp.\ principal curvature spaces) of $M$ at $q$. This follows from the fact that $\xi$ is an $H$-equivariant normal field along principal orbits of a polar action, and therefore $\xi$ is also parallel with respect to the normal connection of the orbits \cite[Corollary~2.3.7]{BCO03}. In particular, the principal curvatures of $M$ along an $H$-orbit are constant. Our purpose is to argue that, generically, $M$ is a strongly $2$-Hopf hypersurface.

Consider the map 
\[
\Phi\colon w\in S^1(T_p\Sigma)\mapsto\langle S_{(\xi_w)_p}(J\xi_w)_p,Jw\rangle\in \R, 
\]
defined in the unit sphere of $T_p\Sigma$, and where $S$ denotes the shape operator of the surface $H\cdot p$, and $\xi_w\in T_p\Sigma$ is the unit vector obtained by rotating $w$ 90 degrees (in some fixed orientation) around the origin of $T_p\Sigma$. Denote by $\g{w}_p$ the subset of vectors of $S^1(T_p \Sigma)$ where $\Phi$ vanishes. Observe that $M=H\cdot \sigma$ is Hopf at $p$ if and only if $\dot{\sigma}(0)\in\g{w}_p$. 

In \cite[end of \S2.2]{DDV}, it was shown that, by virtue of the Ricci equation, the map $\Phi$ cannot vanish identically. Therefore, since $\Phi$ is an analytic map, the set $\g{w}_p$ cannot be infinite. (We note that in~\cite{DDV} it was claimed that $\g{w}_p$ had at most two elements, but this does not need to be true, since $\Phi$ is not linear as asserted there; however, all other statements in~\cite{DDV} remain true.) Thus, if $w=\dot{\sigma}(0)\notin\g{w}_p$, then $M=H\cdot \sigma$ is not Hopf at $p$. By continuity, this implies that, if  $\dot{\sigma}(0)\notin\g{w}_p$ and for $\varepsilon$ small enough, then $M$ is not Hopf at any point. Denote by $\cal{D}$ the rank-$2$ integrable distribution tangent to the $H$-orbits. Then, if $\xi$ is a unit normal vector field to $M$, then $J\xi\in\cal{D}$ (since $\Sigma$ is totally real), and hence, $\cal{D}$ is the smallest distribution of $M$ containing $J\xi$ and invariant under the shape operator of $M$. Moreover, as mentioned above, the principal curvatures of $M$ whose principal curvature spaces lie in $\cal{D}$ are constant along the $H$-orbits, that is, along the integral submanifolds of~$\cal{D}$. This completes the proof that $M$ is strongly $2$-Hopf, whenever $\dot{\sigma}(0)\notin\g{w}_p$. Finally, by the definition of $\g{w}_p$, $H\cdot \sigma$ is Hopf at $p$ if $\dot{\sigma}(0)\in\g{w}_p$.  This concludes the proof of the first part of the Main Theorem.


\subsection{The equations of a strongly $2$-Hopf hypersurface}\label{subsec:strongly2Hopf}

The aim of this subsection is to prove the second part of the Main Theorem, that is, to show that a strongly $2$-Hopf real hypersurface in $\bar{M}^2(c)$, $c\neq 0$, must be locally congruent to a hypersurface constructed as in the previous subsection.

From now on we assume that $M$ is strongly $2$-Hopf with associated distribution $\cal{D}$. We will use the notation given above in Proposition~\ref{prop:algebra}, so in particular $\cal{D}=\mathrm{span}\{U,V\}$. In the following proposition we determine the Levi-Civita connection of $M$.

\begin{proposition}\label{prop:Levi-Civitastrongly}
	The Levi-Civita connection of $M$ in terms of the frame $\{U,V,A\}$ is given by the following equations:
	\begin{align*}
		\nabla_U U  &=-\frac{b(c-4\alpha(\alpha-\beta))}{4a(\alpha-\beta)}A,
		&
		\nabla_V U  &=\frac{c}{4(\alpha-\beta)}A,\\
		\nabla_U V  &=\frac{c}{4(\alpha-\beta)}A,
		&
		 \nabla_V V  &=-\frac{a(c+4\beta(\alpha-\beta))}{4b(\alpha-\beta)}A,\\
		 \nabla_U A  &=\frac{b(c-4\alpha(\alpha-\beta))}{4a(\alpha-\beta)}U
		 -\frac{c}{4(\alpha-\beta)}V,
		 &
		 \nabla_V A  &=-\frac{c}{4(\alpha-\beta)}U
		 +\frac{a(c+4\beta(\alpha-\beta))}{4b(\alpha-\beta)}V\\
		\nabla_A U  &=\left(\frac{c (\beta -\gamma )}{4 (\alpha -\beta )^2}-\frac{c \left(a^2-2 b^2\right)}{4 (\alpha -\beta )}\right)V,	   
		&
		\nabla_A V  &=\left(-\frac{c (\beta -\gamma )}{4 (\alpha -\beta )^2}+\frac{c \left(a^2-2 b^2\right)}{4 (\alpha -\beta )}\right)U,
		\\[1.5ex]
		 \nabla_A A  &=0.
		\end{align*}
		
Furthermore, we have
$\cal{D}a=\cal{D}b=\cal{D}\alpha=\cal{D}\beta=\cal{D}\gamma=0$.	
\end{proposition}

\begin{proof}

First of all, note that, in case that $\gamma$ equals one of the other two principal curvatures in an open set of $M$, then the relations above hold, according to~\cite[Proposition~4.1]{DDV}. Therefore, it is enough to prove Proposition~\ref{prop:Levi-Civitastrongly} for the case where $M$ has three distinct principal curvatures at every point. In particular, Proposition~\ref{prop:Levi-Civita} holds.	

By definition of strongly $2$-Hopf hypersurface, we have $U\alpha=U\beta=V\alpha=V\beta=0$. Then, Equations~\eqref{eq:derivatives1} imply $Ua=Ub=Va=Vb=0$.

Since the distribution $\cal{D}=\spann\{U,V\}$ is integrable due to the strongly 2-Hopf assumption, we must have $\langle \nabla_U V-\nabla_V U,A\rangle=0$. Using Proposition~\ref{prop:Levi-Civita}, this allows us to obtain after some calculations
\begin{equation}\label{eq:Dintegrable}
	\begin{aligned}
		A\alpha &{}= \frac{\alpha  b (\alpha -\gamma )}{a}+\frac{b c (\alpha -\gamma )}{4 a (\beta -\alpha )}+\frac{3 a b c}{4},\\
		A\beta &{}= -\frac{ \beta a (\beta -\gamma )}{b}-\frac{a c (\beta -\gamma )}{4 b (\alpha -\beta )}-\frac{3 a b c}{4},\\
		A b &{}=a \left(\frac{c \left(a^2-2 b^2\right)}{4 (\alpha -\beta )}-\frac{c (\beta -\gamma )}{4 (\alpha -\beta )^2}+\gamma \right).
	\end{aligned}
\end{equation}

The last step is to show that $U\gamma=0$. Proposition~\ref{prop:Levi-Civita}, Equations~\eqref{eq:Dintegrable}, and the assumption $U\alpha=0$, easily imply
	 \begin{align*}
     [U,A]\alpha&=(\nabla_U A-\nabla_A U)\alpha=-\frac{1}{\alpha-\gamma}\left(\frac{\alpha  b (\alpha -\gamma )}{a}+\frac{b c (\alpha -\gamma )}{4 a (\beta -\alpha )}+\frac{3 a b c}{4}\right)U\gamma,\\
	 UA\alpha&=\frac{b(c-4\alpha(\alpha-\beta))}{4a(\alpha-\beta)}U\gamma, \qquad\qquad
	  AU\alpha=0.
	 \end{align*}
Thus, 
\[
0=([U,A]-UA+AU)\alpha=-\frac{3abc}{4(\alpha-\gamma)}U\gamma,	 
\]
which yields $U\gamma=0$, as desired. Finally, by \eqref{eq:derivatives1}, we get $V\gamma=U\gamma=0$. Putting together all these results we obtain Proposition~\ref{prop:Levi-Civitastrongly}.
\end{proof}


In order to conclude the proof of the Main Theorem we need to extract certain geometric information on the integrable distributions $\cal{D}$ and $\R A$, and then use this information to show that $M$ can be constructed as in the statement of the Main Theorem. The arguments needed for this purpose are completely analogous to those developed in~\cite[Sections~5 and~6]{DDV}. Hence, we will restrict ourselves to give a quick idea of the arguments and state the main partial results involved. We refer to \cite{DDV} for detailed proofs.

A first consequence of Propositions~\ref{prop:algebra} and~\ref{prop:Levi-Civitastrongly} is the following.

\begin{proposition}\label{prop:distribution}
	The leaves of the integrable distribution $\cal{D}$ are flat, totally real surfaces of $\bar{M}^2(c)$ with parallel second fundamental form and flat normal bundle.
\end{proposition}

Observe that the relation $\nabla_A A=0$ in Proposition~\ref{prop:Levi-Civitastrongly} implies that the integral curves of $A$ are geodesics of $M$ and, by the Gauss formula, their curvature as curves in $\bar{M}^2(c)$ is~$\gamma$. Moreover, these curves are, locally, intersections of $M$ with totally geodesic, totally real surfaces in $\bar{M}^2(c)$. More precisely, we have:

\begin{proposition}\label{prop:sigma}
	Let $\sigma$ be an integral curve of $A$ through a point $p\in M$. Let $Q_p=\exp_p(\R A_p\oplus\R\xi_p)$, where $\exp_p$ denotes the Riemannian exponential map of $\bar{M}^2(c)$ at $p$. Then $Q_p$ is a totally real, totally geodesic surface of $\bar{M}^2(c)$, and $\sigma$ is contained in $Q_p$.
	
	Furthermore, the curve $\sigma$ is determined by the initial conditions $\sigma(0)=p$, $\dot\sigma(0)=A_p$, and the fact that $\sigma$ is a unit speed curve in $Q_p=\exp_p(\R A_p\oplus\R\xi_p)$ with curvature $\gamma$ with respect to $\xi$.
\end{proposition}

Next, one can show that, if $Q_p$ and $\sigma$ are as above, then $Q_p$ intersects the integral submanifolds of $\mathcal{D}$ perpendicularly along $\sigma$. This, Proposition~\ref{prop:sigma}, the fact that the integral curves of $A$ are geodesics in $M$, and the fact that the curvature $\gamma$ is constant along the leaves of $\cal{D}$, allows to show the following result.

\begin{proposition}\label{prop:equidistant}
	We have:
	\begin{enumerate}[\rm (i)]
		\item The integral surfaces of $\mathcal{D}$ are equidistant submanifolds of $\bar{M}^2(c)$.
		
		\item Let $L$ be an integral surface of the distribution $\mathcal{D}$, and let $L_t$ be an integral surface of $\mathcal{D}$ whose distance to $L$ is a sufficiently small number $t$. Then, in a neighborhood $\mathcal{U}$ of a point in $L$ there exists a parallel normal vector field $\eta_t$ such that
		\[
		L_t=\{\exp_p(\eta_t(p)):p\in \mathcal{U}\}.
		\]
	\end{enumerate}
\end{proposition}


Now, it follows directly from Proposition~\ref{prop:distribution} that the integral submanifolds of $\cal{D}$ are flat, Lagrangian surfaces of $\bar{M}^2(c)$ with parallel mean curvature. Then, \cite[Theorem~2.1]{DDV} guarantees that each one of these surfaces is an open part of a principal orbit of a polar action of cohomogeneity two on $\bar{M}^2(c)$. By Proposition~\ref{prop:equidistant}, the integral surfaces of $\cal{D}$ are obtained by exponentiating a parallel normal vector field along a fixed leaf. Moreover, on a principal orbit of a polar action every parallel normal field is equivariant. Altogether, this implies that all leaves of $\cal{D}$ are principal orbits of the same polar action of a group~$H$. Moreover, for each $p\in M$ the integral curve of $A$ through $p$ is contained in the totally geodesic submanifold $Q_p=\exp_p(\R A_p\oplus\R\xi_p)$, which is perpendicular to the leaf of $\cal{D}$ through $p$ and, then, must be a section for the $H$-action. Therefore, $M$ is obtained, locally, as $H\cdot \sigma$, where $\sigma$ is an integral curve of $A$. This concludes the proof of the Main Theorem.

\section{Austere hypersurfaces}\label{sec:austere}
In this section we investigate austere real hypersurfaces in $\bar{M}^2(c)$, $c\neq 0$, under the only assumption that the Hopf vector field does not have nontrivial projection onto three principal curvature spaces. In other words, we just assume $h\leq 2$. We prove first that $h$ must be constantly equal to $2$ in an open dense subset.

\begin{proposition}\label{prop:austereHopf}
	There are no Hopf austere hypersurfaces in $\bar{M}^2(c)$, $c\neq 0$.
\end{proposition}
\begin{proof}
Austere hypersurfaces have, by definition, vanishing mean curvature. Then, by Theorem~\ref{th:Hopf}, a Hopf austere hypersurface in $\bar{M}^2(c)$, $c\neq 0$, must be an open part of a homogeneous Hopf hypersurface. But by direct inspection of the principal curvatures of the examples in Takagi's and Montiel's lists~\cite[\S3]{NR} one can check that the only Hopf, homogeneous, minimal hypersurfaces in $\bar{M}^2(c)$, $c\neq 0$, are geodesic spheres or tubes around a totally geodesic $\R P^2$ of certain fixed radius. But none of these examples is austere.
\end{proof}

Hence, if $M$ is an austere hypersurface of $\bar{M}^2(c)$, $c\neq 0$, with $h\leq 2$, then there is an open and dense subset of $M$ where $h=2$. In what follows we will assume that calculations take place in this subset.  Note that the assumption that $M$ is austere implies that its principal curvatures are $\alpha$, $-\alpha$ and $0$, for some smooth function $\alpha$ on $M$. We will use the notation established in Proposition~\ref{prop:algebra}. 

\begin{proposition}\label{prop:austereh2}
	Let $M$ be an austere hypersurface of $\bar{M}^2(c)$, $c\neq 0$, with $h=2$, and three distinct principal curvatures $\alpha$, $-\alpha$ and $0$. Then $M$ is strongly 2-Hopf, the Hopf vector field has nontrivial projections onto $T_\alpha$ and $T_{-\alpha}$, and the norm of both projections is $a=b=1/\sqrt{2}$.
\end{proposition}
\begin{proof}
Assume first that $J\xi$ has nontrivial projection onto $T_\alpha$ and $T_{0}$. Thus, we put $\beta=0$ and $\gamma=-\alpha$ in the results of Section~\ref{sec:LeviCivita}. In particular, by \eqref{eq:derivatives1} and \eqref{eq:derivatives2} we have
\begin{align}\label{eq:austere1}
A\alpha&=\frac{b}{4a}(5c+8\alpha^2+9b^2c),& Ab&=\frac{a}{4\alpha}(5c-4\alpha^2-3ca^2), &Va&=Vb=0,
\\ \label{eq:austere1.5}
V\alpha&=-\frac{a}{2b}U\alpha,& Ua&= -\frac{a}{2\alpha} U\alpha,& Ub&=\frac{a^2}{2b\alpha}U\alpha.
\end{align}
Using Proposition~\ref{prop:algebra}, the formulas for the Levi-Civita connection in~Proposition~\ref{prop:Levi-Civita}, \eqref{eq:austere1} and \eqref{eq:austere1.5}, the Gauss equation applied to $(U,V,U,A)$ implies, after some calculations, that
\begin{equation}\label{eq:austere2}
U\alpha=V\alpha=Ua=Ub=0.
\end{equation}
Using again Propositions~\ref{prop:algebra} and~\ref{prop:Levi-Civita}, \eqref{eq:austere1} and \eqref{eq:austere2}, the Gauss equation applied to $(U,V,U,V)$ yields $\alpha^2=\frac{1}{4}(2+3b^2)c$. Similarly, by the Gauss equation applied to $(A,V,U,A)$ we obtain that $\alpha^2 =\frac{(-8+9b^2+27b^4)c}{4(-7+3b^2)}$. But both expressions for $\alpha^2$ are incompatible for $b\in\R$. This contradiction implies the nonexistence of austere hypersurfaces whose Hopf vector field has nontrivial projections onto $T_\alpha$ and $T_{0}$. 	

Since $\alpha$ and $-\alpha$ are interchangeable, we just have to deal with the case where $J\xi$ has nontrivial projection onto $T_\alpha$ and $T_{-\alpha}$. Thus, we put $\beta=-\alpha$ and $\gamma=0$ in the results of Section~\ref{sec:LeviCivita}. Then, by~\eqref{eq:derivatives2} we get 
\begin{equation}\label{eq:austere5}
A\alpha=\frac{ab}{2}(c+4\alpha^2).
\end{equation}
Hence, by applying the Gauss equation to $(A,V,A,U)$, and using Propositions~\ref{prop:algebra} and~\ref{prop:Levi-Civita} with $\beta=-\alpha$, $\gamma=0$, and \eqref{eq:austere5}, we obtain $abc(a^2-b^2)(c+4\alpha^2)=0$. If $a\neq b$ on a nonempty subset $\cal{U}$ of $M$, we deduce that $\cal{U}$ is a real hypersurface with constant principal curvatures $\pm\sqrt{-c}/2$ and $0$ in $\bar{M}^2(c)$, $c<0$. By the classification in~\cite{BD:pams}, $\cal{U}$ must be an open part of a Lohnherr hypersurface, but this example satisfies $a=b$ everywhere, which is a contradiction. Therefore we must have $a=b$ on $M$. Since $a^2+b^2=1$, we deduce that $a=b=1/\sqrt{2}$. But then \eqref{eq:derivatives1} yields $U\alpha=V\alpha=0$. This, together with \eqref{eq:austere5} and Proposition~\ref{prop:Levi-Civita}, implies that $\nabla_U V-\nabla_V U=0$. Hence, $M$ is strongly 2-Hopf, as we wanted to show.
\end{proof} 

In order to conclude the proof of Theorem~\ref{th:austere} we will make use of the notion of ruled hypersurface. Recall that a real hypersurface $M$ in a complex space form is called ruled if the maximal complex distribution $(J\xi)^\perp$ of $M$ is integrable and its integral submanifolds are totally geodesic complex hypersurfaces of the ambient space~\cite[\S8.5.1]{CR}. 

\begin{proof}[Proof of Theorem~\ref{th:austere}]
Observe that $(J\xi)^\perp=\R JA\oplus\R A=\R(bU-aV)\oplus\R A$. By Proposition~\ref{prop:austereh2} we have $SJA=\alpha b U+\alpha a V=(\alpha/\sqrt{2})(U+V)=\alpha J\xi$ and $SA=0$, which implies that $S(J\xi)^\perp\subset \R J\xi$. By \cite[Proposition~8.27]{CR}, $M$ is a ruled hypersurface. In particular, $M$ is a minimal ruled hypersurface in $\bar{M}^2(c)$, $c\neq 0$. Lohnherr and Reckziegel~\cite{LR} proved that there is at most one minimal ruled hypersurface in $\C P^2$ up to local congruence, and at most three in $\C H^2$.

Kimura~\cite{Ki:mathann} proved that a cone over a Clifford torus in $\C P^2$ is austere and ruled. Since ruled hypersurfaces satisfy $h\leq 2$ everywhere (indeed $h=2$ on an open and dense subset), Kimura's example gives the only possibility of an austere hypersurface with $h\leq 2$ in $\C P^2$. 

In $\C H^2$ there are three known (noncongruent) examples of minimal ruled hypersurfaces: Clifford cones~\cite[\S3]{ALS}, Lohnherr hypersurfaces~\cite[\S4]{LR}, and bisectors~\cite[p.~447]{CR}; see also~\cite[p.~253]{Goldman}. All of them are known to be austere with $h\leq 2$. Therefore, these are precisely the examples of austere hypersurfaces with $h\leq 2$ in $\C H^2$.
\end{proof}

\begin{remark}
It is known that a ruled hypersurface $M$ in a complex space form is locally constructed by attaching to an integral curve $\tau$ of $J\xi$ the complex totally geodesic hypersurfaces that are normal to $\dot{\tau}$. It was also shown in \cite{LR} that a ruled hypersurface in $\bar{M}^2(c)$, $c\neq 0$, is minimal if and only if $\tau$ is a circle contained in a totally geodesic and totally real submanifold of $\bar{M}^2(c)$.  Moreover, in the projective case, any two such circles give rise to the same ruled hypersurface, up to congruence, whereas in the hyperbolic case, two circles $\tau_1$, $\tau_2$ give rise to congruent ruled hypersurfaces if and only if their curvatures $\left\|\bar{\nabla}_{\dot{\tau_1}}\dot{\tau_1}\right\|$, $\left\|\bar{\nabla}_{\dot{\tau_2}}\dot{\tau_2}\right\|$ are both greater, equal, or less than $\sqrt{-c}/2$. It follows from our study above that $\bar{\nabla}_{J\xi}J\xi=\alpha A$ for an austere hypersurface with $h=2$. Note that from~\eqref{eq:austere5} we have that $\alpha-\sqrt{-c}/2$ has constant sign. One can show that the cases $\alpha>\sqrt{-c}/2$, $\alpha=\sqrt{-c}/2$ and $\alpha<\sqrt{-c}/2$ correspond, respectively, to Clifford cones, Lohnherr hypersurfaces and bisectors.
\end{remark}

We conclude this section by observing that we could have finished the proof of Theorem~\ref{th:austere} without using the results about ruled hypersurfaces. We sketch briefly the idea of this alternative argument. 

In view of Propositions~\ref{prop:austereHopf} and~\ref{prop:austereh2}, an austere hypersurface $M$ in $\bar{M}^2(c)$, $c\neq 0$, satisfying $h\leq 2$, is strongly $2$-Hopf in an open and dense subset, and thus, must be constructed by the procedure described in Subsection~\ref{subsec:construction}. Moreover, according to Proposition~\ref{prop:sigma}, the curve $\sigma$ inside the section $\Sigma$ of a polar $H$-action must have curvature $\gamma=0$. In other words, we need $\sigma$ to be a pregeodesic in $\Sigma$, that is, $\bar{\nabla}_{\dot{\sigma}}\dot{\sigma}\in \mathrm{span}\{\dot{\sigma}\}$.

Moreover, if $M=H\cdot\sigma$ is to be austere, the trace of the shape operator $S_\xi$ of the orbits $H\cdot \sigma(t)$, with respect to the normal vector field $\xi$ of $M$, must vanish. This follows from the fact that the principal curvatures of the integral leaves of $\cal{D}=\textrm{span}\{U,V\}$ with respect to $\xi$ coincide with the spectrum of the shape operator of $M$ restricted to $\cal{D}$, but this spectrum is $\{\alpha,-\alpha\}$, for some $H$-invariant function $\alpha$ on $M$, according to Proposition~\ref{prop:austereh2}. Thus, $\xi_p$ must be perpendicular to the mean curvature vector field of the orbit $H\cdot p$, for every $p\in M$. By $H$-equivariance, it is enough to have this property along the points of $\sigma$. Let $\cal{H}$ be the vector field on $\Sigma$ defined by the fact that $\cal{H}_p$ is the mean curvature vector of $H\cdot p$ at $p$. Then the condition reads $\cal{H}_{\sigma(t)}\in\mathrm{span}\{\dot{\sigma}(t)\}$ for every $t$. 

It turns out that $M=H\cdot \sigma$ is austere if and only if $\sigma$ is a pregeodesic and $\cal{H}_{\sigma(t)}\in\mathrm{span}\{\dot{\sigma}(t)\}$, for all $t$. Thus, the idea is to find all curves $\sigma$ satisfying both conditions, for each polar action of cohomogeneity $2$ on $\bar{M}^2(c)$, $c\neq 0$. This requires a good understanding of such actions and, in particular, one needs to determine the mean curvature vector field $\cal{H}$ on $\Sigma$ explicitly, and to compute the derivative $\bar{\nabla}_{\cal{H}}\cal{H}$. Here we skip the elementary but long calculations involved. 

As mentioned in Subsection~\ref{subsec:polar}, the unique polar action of cohomogeneity two on $\C P^2$ (up to orbit equivalence) is the action of $H=U(1)\times U(1)$. A section $\Sigma$ for this action is a totally geodesic $\R P^2$, and the orbit space $\C P^2/H$ is homeomorphic to a geodesic triangle of angles $(\pi/2,\pi/2,\pi/2)$ inside $\Sigma=\R P^2$. Due to the action of the Weyl group on $\Sigma$, it is enough to find a curve $\sigma$ in this triangle. It turns out that the only curves satisfying the above mentioned conditions are the bisectors of the three angles of the triangle. Each such a curve $\sigma$ joins a vertex of the triangle (which is a fixed point of the $H$-action) to the only minimal principal $H$-orbit, which we call a Clifford torus of $\C P^2$. Thus, the resulting hypersurface $H\cdot \sigma$ is a cone over a Clifford torus in $\C P^2$.

In $\C H^2$ there are four cohomogeneity-two polar actions up to orbit equivalence. A section $\Sigma$ for each of them is always a totally geodesic $\R H^2$. The action of $H=U(1)\times U(1)$ is in some sense dual to the corresponding action on $\C P^2$, and the only suitable curves $\sigma$ in $\Sigma$ give rise to Clifford cones (cf.~\cite[\S3.2]{GG}). The action of the group $H$ with Lie algebra $\g{h}=\g{g}_0$ admits only one curve $\sigma$ in $\Sigma$ such that $H\cdot \sigma$ is austere; in this case, the hypersurface $H\cdot \sigma$ turns out to be a bisector (cf.~\cite[\S3.4]{GG}).  The action corresponding to $\g{h}=\g{k}_0\oplus\g{g}_{2\alpha}$ does not admit any suitable curve $\sigma$. Finally, the case $\g{h}=\ell\oplus\g{g}_{2\alpha}$, with $\ell$ a 1-dimensional subspace of $\g{g}_\alpha$, admits only one suitable curve $\sigma$, and the corresponding austere hypersurface $H\cdot \sigma$ is a Lohnherr hypersurface (which is called a fan in~\cite[\S3.6]{GG}).

\section{Applications}\label{sec:aplications}

In this section we derive some characterizations of strongly 2-Hopf hypersurfaces that satisfy some additional properties. We first prove Corollaries~\ref{cor:CMC} and~\ref{cor:Levi-flat} in Subsections~\ref{subsec:strongly2HopfCMC} and~\ref{subsec:strongly2HopfLevi-flat}, respectively, and then, in Subsection~\ref{subsec:strongly2HopfCMCLevi-flat} we classify strongly 2-Hopf, Levi-flat real hypersurfaces with constant mean curvature in $\bar{M}^2(c)$, $c\neq 0$ (Theorem~\ref{th:strongly2HopfCMCLevi-flat}).

The Main Theorem guarantees that strongly 2-Hopf hypersurfaces in $\bar{M}^2(c)$, $c\neq 0$, are constructed locally as a set $H\cdot \sigma = \{h(\sigma(t)): t\in(-\varepsilon, \varepsilon), \, h\in H\}$, where $H$ is a connected group of isometries acting polarly and with cohomogeneity two on $\bar{M}^2(c)$, and $\sigma$ is a smooth curve in the regular part of a section $\Sigma$ of the $H$-action. Our purpose is to determine which curves $\sigma$ give rise to a real hypersurface with one or several additional properties.

\subsection{Strongly 2-Hopf hypersurfaces with constant mean curvature}\label{subsec:strongly2HopfCMC}

In order to prove Corollary~\ref{cor:CMC}, we assume that the mean curvature of the resulting hypersurface $H\cdot\sigma$ is constant. Thus, let $p\in \Sigma$ be a regular point, $w\in T_p\Sigma$ a tangent vector, and $\sigma$ a smooth curve in the regular part of $\Sigma$ such that $\sigma(0)=p$ and $\dot{\sigma}(0)=w$. Let $\xi$ be one of the two unit normal vector fields along $\sigma$ that are tangent to $\Sigma$, and let $\gamma$ denote the curvature of $\sigma$ with respect to $\xi$. We also denote by $\xi$ the unique extension to a smooth unit normal vector field along $H\cdot \sigma$; note that such extension is $H$-equivariant. Observe also that, by equivariance, the principal curvatures of $H\cdot \sigma$ with respect to $\xi$ are constant along each $H$-orbit. Then the mean curvature of $H\cdot \sigma$ with respect to $\xi$ will have a constant value $\eta\in \R$ if and only if the curvature function $\gamma$ satisfies $\gamma(t)=\eta-\alpha(\xi(t))-\beta(\xi(t))$ for all $t$ where $\sigma$ is defined, being $\alpha(\xi(t))$ and $\beta(\xi(t))$ the principal curvatures of the orbit $H\cdot \sigma(t)$ with respect to $\xi_{\sigma(t)}$ at the point $\sigma(t)$. In other words, we need $(\bar{\nabla}_{\dot{\sigma}}\dot{\sigma})(t)=(\eta-\alpha(\xi(t))-\beta(\xi(t)))\xi_{\sigma(t)}$ for all $t$. But, in local coordinates, this yields an ordinary differential equation of second order in normal form, so it admits a unique local solution $\sigma$ for initial conditions $\sigma(0)=p$ and $\dot{\sigma}(0)=w$. This, together with the Main Theorem, proves Corollary~\ref{cor:CMC}. Observe that, by the Main Theorem, the hypersurface with constant mean curvature constructed above is generically strongly 2-Hopf.

\subsection{Levi-flat strongly 2-Hopf hypersurfaces}\label{subsec:strongly2HopfLevi-flat}

The Levi form of a real hypersurface $M$ in a K\"ahler manifold is the symmetric bilinear map $L\colon (J\xi)^\perp \times (J\xi)^\perp \rightarrow \nu M$ defined by
\[
	L(X,Y)=\II(X,Y)+\II(JX,JY), 
\]
where $(J\xi)^\perp$ is the maximal complex distribution of $M$. Then $M$ is called Levi-flat if its Levi form vanishes identically. It is easy to check that $M$ is Levi-flat if and only if the maximal complex distribution of $M$ is integrable. Thus, ruled hypersurfaces are a very particular case of Levi-flat hypersurfaces. See \cite{He64} for more information on Levi-flat hypersurfaces.

Consider a real hypersurface $M$ in $\bar{M}^2(c)$, $c\neq 0$, satisfying $h=2$. 
We will use the notation established in Section~\ref{sec:LeviCivita}. Assume 
that $M$ is Levi-flat. Then, its Levi-form vanishes. Since $A, JA \in 
(J\xi)^\perp$ by Proposition~\ref{prop:algebra}, we have 
$\II(A,A)+\II(JA,JA)=0$ or, equivalently, $\langle S A, A\rangle+\langle SJA, 
JA\rangle=0$. Using Proposition~\ref{prop:algebra} again, this condition reads
\[
\gamma+b^2\alpha+a^2\beta=0.
\]

Now, in order to prove Corollary~\ref{cor:Levi-flat} the procedure is analogous to the one described above in~\S\ref{subsec:strongly2HopfCMC}. One just has to take into account that now the curve $\sigma$ must have curvature function $\gamma(t)=-b(\xi(t))^2\alpha(\xi(t))-a(\xi(t))^2\beta(\xi(t))$, where $a(\xi(t))$ and $b(\xi(t))$ are the norms of the orthogonal projections of $J\xi(t)$ onto the principal curvature spaces $T_{\alpha(\xi(t))}$ and $T_{\beta(\xi(t))}$ of the surface $H\cdot \sigma(t)$ with respect to $\xi(t)$, at each point $\sigma(t)$.

\subsection{Levi-flat strongly 2-Hopf hypersurfaces with constant mean curvature}\label{subsec:strongly2HopfCMCLevi-flat}
Our aim in this subsection is to prove Theorem~\ref{th:strongly2HopfCMCLevi-flat}.

Let $M$ be a Levi-flat strongly $2$-Hopf real hypersurface in $\bar{M}^2(c)$, $c\neq 0$, with constant mean curvature $\eta$. By Subsections~\ref{subsec:strongly2HopfCMC} and~\ref{subsec:strongly2HopfLevi-flat}, we have that $\gamma=\eta-\alpha-\beta$ and $\gamma=-b^2\alpha-a^2\beta$. Since $a^2+b^2=1$, we deduce $a^2\alpha+b^2\beta=\eta$.
If we take derivatives in this expression with respect to the vector field $A$ we obtain 
\[
2a\alpha Aa+a^2A\alpha+2b\beta Ab+b^2Ab=0. 
\]
From \eqref{eq:Dintegrable} in Section \ref{sec:strongly2Hopf} and from the relation $aAa+bAb=0$, we deduce the expressions of $Aa$, $Ab$, $A\alpha$ and $A\beta$ in terms of $a$, $b$ and the principal curvatures. Thus, substituting into the previous equation we obtain after some calculations that $3\gamma=\alpha+\beta$. Since $\gamma=\eta-\alpha-\beta $, then $\gamma=\eta/4$ and $\alpha+\beta=3\eta/4$.

From the equations $a^2+b^2=1$ and $\alpha b^2+\beta a^2=-\gamma$ we get the expressions $a^2=(\gamma+\beta)/(\beta-\alpha)$ and $b^2=(\alpha+\gamma)/(\alpha-\beta)$. Since $\alpha+\beta$ is constant, we have $A(\alpha+\beta)=0$. Putting together this with \eqref{eq:Dintegrable}, the previous expressions for $a^2$ and $b^2$, and the relations $\gamma=\eta/4$ and $\beta=3\eta/4-\alpha$, we obtain after some calculations that
\[
0=\eta(8\alpha^2-6\eta\alpha+3\eta^2-4c).
\]

We distinguish between the minimal and non-minimal cases. Thus, if $\eta\neq 0$, the previous equation implies that $\alpha$ is constant, and then $M$ has constant principal curvatures. But real hypersurfaces with constant principal curvatures in the complex projective and hyperbolic planes have been classified \cite{BD:pams}, \cite{Wa83} (see~\cite{Do:dga} for a survey). On the one hand, in $\C P^2$ there do not exist non-Hopf hypersurfaces with constant principal curvatures. On the other hand, in $\C H^2$ the only non-Hopf hypersurfaces with constant principal curvatures are the Lohnherr hypersurface (which is minimal), and its equidistant hypersurfaces (which are non-minimal). All of them are strongly 2-Hopf, as follows from~\cite[\S4.1]{BD09} (cf.~\cite{DD:indiana}). However, only the Lohnherr hypersurface is Levi-flat: it is the only one that satisfies the relation $\gamma=-b^2\alpha-a^2\beta$, as can be checked from~\cite[Theorem~3.12]{DD:indiana}. Hence, the case $\eta\neq 0$ is impossible.

Assume now that $\eta=0$. Then $\gamma=0$, $\beta=-\alpha$ and $a^2=b^2=1/2$. In particular, $M$ is an austere strongly 2-Hopf hypersurface in $\bar{M}^2(c)$, $c\neq 0$. By the classification achieved in Section~\ref{sec:austere}, we deduce that $M$ must be an open part of a Lohnherr hypersurface, or a Clifford cone, or a bisector. Finally, observe that all these examples are Levi-flat, since they are ruled. This concludes the proof of Theorem~\ref{th:strongly2HopfCMCLevi-flat}.

\end{document}